\documentclass[a4paper,11pt,twoside]{amsart}

\usepackage[latin1]{inputenc}
\usepackage{amsmath, amsfonts, amssymb,amsthm}
\usepackage[mathscr]{eucal} 

\usepackage{a4wide}
\usepackage[pdftex]{graphicx}
\usepackage{palatino} 
\usepackage{mathpazo} 

\usepackage{enumerate}

\newcommand{\N}{\mathbb{N}}
\newcommand{\R}{\mathbb{R}}
\newcommand{\C}{\mathbb{C}}
\newcommand{\s}{\mathbb{S}}
\newcommand{\h}{\mathbb{H}}
\newcommand{\E}{\mathrm{E}}
\newcommand{\Sl}{\mathrm{Sl}}

\newcommand{\pRe}{\mathrm{Re}}
\newcommand{\pIm}{\mathrm{Im}}
\newcommand{\df}{\,\mathrm{d}}
\newcommand{\rot}{\mathrm{Rot}}

\newcommand{\prodesc}[2]{\left\langle #1, #2 \right \rangle}
\newcommand{\abs}[1]{\left\lvert #1 \right\rvert}
\DeclareMathOperator{\arctanh}{arctanh}
\DeclareMathOperator{\arcsinh}{arcsinh}

\newtheorem{theorem}{Theorem}
\newtheorem{proposition}{Proposition}
\newtheorem{corollary}{Corollary}
\newtheorem{lemma}{Lemma}

\theoremstyle{definition}

\theoremstyle{remark}
  \newtheorem{remark}{Remark}

\numberwithin{equation}{section}

\title[Rotationally invariant CMC surfaces in homogeneous $3$-manifolds]{Rotationally invariant constant mean curvature surfaces in homogeneous $3$-manifolds}

\author{Francisco Torralbo}
\address{Departamento de Geometr\'{\i}a  y Topolog\'{\i}a \\
Universidad de Granada \\
18071 Granada, SPAIN} \email{ftorralbo@ugr.es}

\thanks{Research partially supported by a MCyT-Feder research project MTM2007-61775 and a Junta Andalucía  Grant P06-FQM-01642.}



\begin{document}
  \begin{abstract}
   We classify constant mean curvature surfaces invariant by a $1$-parameter group of isometries in the Berger spheres and in the special linear group $\Sl(2,\R)$. In particular, all constant mean curvature spheres in those spaces are described explicitly, proving that they are not always embedded. Besides new examples of Delaunay-type surfaces are obtained. Finally the relation between the area and volume of these spheres in the Berger spheres is studied, showing that, in some cases, they are not solution to the isoperimetric problem.
  \end{abstract}
  
\maketitle

\section{Introduction}
In the last years, constant mean curvature surfaces of the homogeneous Riemannian $3$-manifolds have been deeply studied. The starting point was the work of Abresch and Rosenberg~\cite{ARb}, where they found a holomorphic quadratic differential in any constant mean curvature surface of a homogeneous Riemannian $3$-manifold with isometry group of dimension $4$. Berger spheres, the Heisenberg group, the special linear group $\Sl(2,\R)$ and the Riemannian product $\s^2 \times \R$ and $\h^2 \times \R$, where $\s^2$ and $\h^2$ are the $2$-dimensional sphere and hyperbolic plane, are the most relevant examples of such homogeneous $3$-manifolds.
\medskip

Abresch and Rosenberg~\cite{ARa} proved that a complete constant mean curvature surface in $\s^2 \times \R$ and $\h^2 \times \R$ with vanishing Abresch-Rosenberg differential must be \emph{rotationally} invariant (that is, invariant under a $1$-parameter group of isometries acting trivially on the fiber). Moreover, do Carmo and Fernández~\cite[Theorem 2.1]{dCF} showed that, even locally, every constant mean curvature surface in $\s^ 2 \times \R$ or $\h^2 \times \R$ with vanishing Abresch-Rosenberg differential must be rotationally invariant too. Finally, Espinar and Rosenberg~\cite{ER} for every homogeneous Riemannian spaces with isometry group of dimension $4$, proved that every constant mean curvature surface with vanishing Abresch-Rosenberg differential must be invariant by a $1$-parameter group of isometries.
\medskip

Constant mean curvature surfaces invariant by a $1$-parameter group of isometries were studied in the product spaces $\s^2 \times \R$ and $\h^2 \times \R$ by Hsiang and Hsiang~\cite{HH89} and Pedrosa and Ritoré~\cite{PR99}. Also, in the Heisenberg group, the study was made by Tomter~\cite{To}, Figueroa, Mercuri and Pedrosa~\cite{FMP} and Caddeo, Piu and Ratto~\cite{CPR}. Tomter described explicitly in~\cite{To} the constant mean curvature spheres computing their volume and area in order to give an upper bound for the isoperimetric profile of the Heisenberg group. The authors in~\cite{FMP} studied not only the \emph{rotationally} invariant case, but the surfaces invariant by any closed $1$-parameter group of isometries of the Heisenberg group, and organized most of the results that had appeared in the literature. In the special linear group $\Sl(2,\R)$ the classification was obtained by Gorodsky~\cite{G} and, very recently, the classification was made in the universal cover of $\Sl(2,\R)$ by Espinoza~\cite{E}.
\medskip

The aim of this paper is to classify the constant mean curvature surfaces invariant by a $1$-parameter group of isometries that fix a curve, that is, \emph{rotationally invariant}, in the Berger spheres (Theorem~\ref{tm:clasificacion-berger}). 
In this classification it turns out that constant mean curvature spheres are not always embeddded (see figure~\ref{fig:immersed-cmc-sphere}) contradicting the result announced by Abresch and Rosenberg in~\cite[Theorem 6]{ARb}. Besides, we obtain some new examples of surfaces similar to the Delaunay constant mean curvature surfaces in $\R^3$. Moreover, since we obtain an explicit immersion for the constant mean curvature sphere (see Cororally~\ref{cor:sphere-CMC-berger}), we analyse the relation between the area and the volume of the constant mean curvature spheres and show that, for some Berger spheres, they are not the best candidates to solve the isoperimetric problems. Finally some Delaunay-type surfaces give rise, in some Berger spheres, to embedded minimal tori which are not the Clifford torus, proving that the Lawson conjecture is not true in some Berger spheres (see Remark~\ref{rm:clasificacion-berger}.(2)).
\medskip

Using the same techniques, and giving a sketch of the proofs, we classify rotationally invariant constant mean curvature surfaces in $\Sl(2,\R)$ (see Theorem~\ref{tm:clasificacion-Sl(2,R)}), and we obtain an explicit description for the constant mean curvature spheres, showing that they are not always embedded (see figure~\ref{fig:sl-embebidas}). Although the classification in $\Sl(2,\R)$ was made by Gorodsky in~\cite{G}, there exist a mistake in \cite[Theorem 2.(b)]{G} where he claims that for every $H > 0$ there exists a sphere with  constant mean curvature $H$, something that is actually false (see Remark~\ref{rm:clasificacion-Sl(2,R)}.(1)).


\section[CMC surfaces in homogeneous spaces]{Constant mean curvature surfaces in the homogeneous spaces}
Let $N$ be a homogeneous Riemannian $3$-manifold with isometry group of dimension $4$. Then there exists a Riemannian submersion $\Pi:N\rightarrow M^2(\kappa)$, where $M^2(\kappa)$ is a 2-dimensional simply connected space form of constant curvature $\kappa$, with totally geodesic fibers and there exists a unit Killing field $\xi$ on $N$ which is vertical with respect to $\Pi$. We will assume that $N$ is oriented, and we can define a cross product $\wedge$, such that if $\{e_1,e_2\}$ are linearly independent vectors at a point $p$, then $\{e_1,e_2,e_1\wedge e_2\}$ is the orientation at $p$. If $\bar{\nabla}$ denotes the Riemannian connection on $N$, the properties of $\xi$ imply (see \cite{D}) that for any vector field $V$
\begin{equation}\label{eq:killing-property}
\bar{\nabla}_V\xi=\tau(V\wedge\xi),
\end{equation}
where the constant $\tau$ is the bundle curvature. As the isometry group of $N$ has dimension 4, $\kappa-4\tau^2\not=0$. The case $\kappa-4\tau^2=0$ corresponds to $\s^3$ with its standard metric if $\tau\not=0$ and to the Euclidean space $\R^3$ if $\tau=0$, which have isometry groups of dimension $6$.

 In our study we are going to deal mainly with the Berger spheres, which correspond to $\kappa > 0$ and $\tau \neq 0$, and with the special linear group $\Sl(2, \R)$, which correspond to $\kappa < 0$ and $\tau \neq 0$. The fibration in both cases is by circles.

\begin{quote}
\emph{Along the paper $\E(\kappa,\tau)$ will denote an oriented homogeneous Riemannian $3$-manifold with isometry group of dimension $4$, where $\kappa$ is the curvature of the basis, $\tau$ the bundle curvature (and therefore $\kappa-4\tau^2\not=0$).}
\end{quote}

Now, let $\Phi:\Sigma\rightarrow \E(\kappa, \tau)$ be an immersion of an orientable surface $\Sigma$ and $N$ a unit normal vector field. We define the function $C:\Sigma\rightarrow \R$ by
\[
C=\langle N,\xi\rangle,
\]
where $\langle,\rangle$ denotes the metric in $\E(\kappa, \tau)$, and also the metric of $\Sigma$. It is clear that $C^2\leq 1$.
\medskip

Suppose now  that the immersion $\Phi$ has {\em constant mean curvature}. Consider on $\Sigma$ the structure of Riemann surface associated to the induced metric and let $z=x+iy$ be a conformal parameter on $\Sigma$. Then, the induced metric is written as $e^{2u}|dz|^2$ and we denote by $\partial_z=(\partial_x-i\partial_y)/2$ and $\partial_{\bar{z}}=(\partial_x+i\partial_y)/2$ the usual operators.

For these surfaces, the Abresch-Rosenberg quadratic differential $\Theta$, defined by
\[
\Theta(z)=\left(\langle\sigma(\partial_z,\partial_z),N\rangle-
\frac{(\kappa-4\tau^2)}{2(H+i\tau)}\langle\Phi_z,\xi\rangle^2\right)(dz)^2,
\]
where $\sigma$ is the second fundamental form of the immersion, is holomorphic (see \cite{ARb}). We denote  $p(z)=\langle \sigma(\partial_z,\partial_z),N\rangle$ and $A(z)=\langle\Phi_z,\xi\rangle$.

\begin{proposition}[\cite{D, FM}] \label{prop:integrability-conditions}
The fundamental data $\{u,C,H,p,A\}$ of a constant mean curvature immersion $\Phi:\Sigma\rightarrow \E(\kappa, \tau) $  satisfy the following integrability conditions:
\begin{equation}\label{eq:integrability-conditions}
\begin{aligned}
p_{\bar{z}}&=\frac{e^{2u}}{2}(\kappa - 4\tau^2)CA,&
A_{\bar{z}}&=\frac{e^{2u}}{2}(H+i\tau)C,\\
C_z&=-(H-i\tau)A-2e^{-2u}\bar{A}p,&
|A|^2&=\frac{e^{2u}}{4}(1-C^2).
\end{aligned}
\end{equation}
Conversely, if $u,C:\Sigma\rightarrow\R$ with $-1\leq C\leq 1$ and $p,A:\Sigma\rightarrow\C$ are functions on a simply connected surface $\Sigma$ satisfying equations \eqref{eq:integrability-conditions}, then there exists a unique, up to congruences, immersion $\Phi:\Sigma\rightarrow\E(\kappa, \tau)$ with constant mean curvature $H$ and whose fundamental data are $\{u,C,H,p,A\}$.
\end{proposition}

Given a constant mean curvature surface $\Sigma$ with vanishing Abresch-Rosenberg differential we know that it must be invariant by a $1$-parameter group of isometries (see~\cite{ARa, dCF, ER}).

Now we will restrict our attention to constant mean curvature spheres $S$, which will be treated in a uniform way for all $\E(\kappa, \tau)$. The advantage of using this aproach is that we will obtain a global formula for the area of the constant mean curvature spheres in terms of $\kappa$ and $\tau$ (see Proposition~\ref{prop:area-esferas}).  In this case using~\eqref{eq:integrability-conditions} and taking into account that $\Theta = 0$ we get
\[
\begin{split}
 C_z &= \frac{-(H-i\tau)A}{4(H^2 + \tau^2)}[4(H^2 + \tau^2) + (\kappa - 4\tau^2)(1-C^2)], \\ 
 C_{z\bar{z}} &= \frac{-e^{2u}}{32(H^2 + \tau^2)}C[4(H^2 + \tau^2) + (\kappa - 4\tau^2)(1-C^2)]^2
\end{split}
\]
Because $[4(H^2 + \tau^2) + (\kappa - 4\tau^2)(1-C^2)] (> 4H^2 + \kappa) > 0$ the only critical points of $C$ appear where $A$ vanish, i.e., taking into account~\eqref{eq:integrability-conditions} when $C^2(p) = 1$. But the Hessian of $C$ is given by $(H^2 + \tau^2)^2 > 0$ (except for minimal spheres in $\s^2 \times \R$, but in that case the sphere is the slice $\s^2 \times \{t_0\} \subset \s^2 \times \R$) so all critical points are non degenerate. Hence, $C$ is a Morse function on $S$ and so it has only two critical points $p$ and $q$ which are the absolute maximum and minimum of $C$. The function $v: S \rightarrow \R$ given by $v = \arctanh C$ is a harmonic function from~\eqref{eq:integrability-conditions} with singularities at $p$ and $q$ and without critical points. Now there exist a global conformal parameter $w = x + i y$ over $S$ such that $v(w) = \pRe(w) = x$. In this new global conformal parameter the function $C$ is $C(x) = \tanh(x)$ and so it is not difficult to check that the conformal factor of the metric can be written as:
\[
 e^{2u(x)} = \frac{16(H^2 + \tau^2) \cosh^2 x}{[4(H^2 + \tau^2)\cosh^2 x + (\kappa - 4\tau^2)]^2}, \qquad x \in \R
\]

Now, to obtain the area of the constant mean curvature sphere it is sufficient to integrate the above function for $x \in \R$ and $y \in [0, T]$, where $T$ must be $2\pi$ since, by the Gauss-Bonnet theorem,
\[
 4\pi = \int_S K \df A = T \int_{\R} e^{2u(x)} K \df x = -T \int_{\R} u''(x) \df x = 2T.
\]
Then, the area is given by
\[
\textrm{Area}(S) = \int_0^{2\pi} \int_\R e^{2u(x)}\,\mathrm{d}x\,\mathrm{d}y = 2\pi\int_\R e^{2u(x)}\,\mathrm{d}x
\]
and a straightforward computation yields the following lemma.

\begin{proposition}\label{prop:area-esferas}
 The area of a constant mean curvature sphere $S$ in $\E(\kappa, \tau)$ is given by:
\[
	\textrm{Area}(S) = 
	\begin{cases}
	 	\displaystyle{\frac{8\pi}{4H^2 + \kappa} \left[ 1 + \frac{4(H^2 + \tau^2)}{\sqrt{4H^2 + \kappa}\sqrt{4\tau^2 - \kappa}} \arctan\left(\frac{\sqrt{4\tau^2 - \kappa}}{\sqrt{4H^2 + \kappa}} \right)\right]}, & \text{if } \kappa - 4\tau^2 < 0, \\ \\
		\displaystyle{\frac{8\pi}{4H^2 + \kappa} \left[ 1 + \frac{4(H^2 + \tau^2)}{\sqrt{4H^2 + \kappa}\sqrt{\kappa - 4\tau^2}} \arctanh\left(\frac{\sqrt{\kappa - 4\tau^2}}{\sqrt{4H^2 + \kappa}} \right)\right]}, & \text{if } \kappa - 4\tau^2 > 0
	\end{cases}
		\label{eq:area-CMC-sphere}
	\]
	where $H$ is the mean curvature of $S$.
\end{proposition}

\begin{remark}
 The same formula was already obtained for constant mean curvature spheres in the Heisenberg group with $\kappa = 0$ and $\tau = 1$ by~\cite[Proposition 5]{To} and in $M^2(\kappa) \times \R$ by~\cite{Pedrosa04} when $\kappa = 1$ and by~\cite{HH89} when $\kappa = -1$. It is important to remark that in \cite{To, Pedrosa04} the mean curvature is the trace of the second fundamental form while here the mean curvature is half of it.
\end{remark}

\section{The berger spheres}

A Berger sphere is a usual $3$-sphere $\s^3 = \{(z, w) \in \C^2:\, \abs{z}^2 + \abs{w}^2 = 1\}$ endowed with the metric
\[
 g(X, Y) = \frac{4}{\kappa}\left[\prodesc{X}{Y} + \left(\frac{4\tau^2}{\kappa} - 1\right)\prodesc{X}{V}\prodesc{Y}{V} \right]
\]
where $\prodesc{\,}{\,}$ stands for the usual metric on the sphere, $V_{(z, w)} = (iz, iw)$, for each $(z, w) \in \s^3$ and $\kappa$, $\tau$ are real numbers with $\kappa > 0$ and $\tau \neq 0$. For now on we will denote the Berger sphere $(\s^2,g)$ as $\s^3_b(\kappa, \tau)$, which is a model for a homogeneous space $\E(\kappa, \tau)$ when $\kappa > 0$ and $\tau \neq 0$. In this case the vertical Killing field is given by $\xi = \frac{\kappa}{4\tau}V$. We note that $\s^3_b(4, 1)$ is the round sphere.

The group of isometries of $\s^3(\kappa, \tau)$ is $U(2)$. The next proposition classifies, up to conjugation, the $1$-parameter groups of $U(2)$ into two types. 

\begin{proposition}\label{prop:classificacion-subgroups-sphere}
A $1$-parameter group of $U(2)$, up to conjugation and reparametrization, must be one of the following types:
\begin{enumerate}[(i)]
	\item $\mathrm{Rot} = \left\{\begin{pmatrix}
									 1	&	0	\\
										0			& e^{it}
									\end{pmatrix}:\, t \in \R\right\}$
	\item $\left\{\begin{pmatrix}
									 e^{i \alpha t}	&	0	\\
										0			& e^{it}
									\end{pmatrix}:\, t \in \R\right\}$, with $\alpha \in \R\setminus\{0\}$.
\end{enumerate}
\end{proposition}

\begin{proof}
All $1$-parametric group of U(2) are generated, via the exponential map, by an element of the Lie algebra 
\[
\mathfrak{u}(2) = \left\{
					\begin{pmatrix}
						ia	&	x e^{iy}	\\
						-xe^{-iy}	&	ib
					\end{pmatrix}:\, a, b, x, y \in \R
				 \right\}
\]
We are going to reduce the possible $1$-parametric groups by conjugation. It is clear that given $A \in \mathfrak{u}(2)$ and $D \in U(2)$ then $A$ and  $D.A.D^{-1}$ are conjugated. So if $A = \left(\begin{smallmatrix}ia	&	x e^{iy}	\\
						-xe^{-iy}	&	ib\end{smallmatrix} \right)$, 
then taking $D = \left(\begin{smallmatrix}
1	&	0	\\
0	&	e^{iy}
\end{smallmatrix} \right)$ it follows that
\[
\begin{pmatrix}
1	&	0	\\
0	&	e^{iy}
\end{pmatrix}
\begin{pmatrix}
ia	&	x e^{iy}	\\
-xe^{-iy}	&	ib
\end{pmatrix}
\begin{pmatrix}
1	&	0	\\
0	&	e^{-iy}
\end{pmatrix}
=
\begin{pmatrix}
ia	&	x	\\
-x	&	ib
\end{pmatrix}
\]
Hence we may suppose that, up to conjugation, $y = 0$, i.e., $A = \left(\begin{smallmatrix} ia	&	x	\\ -x	&	ia\end{smallmatrix}\right)$. First, if $a = b$, taking $D = \frac{1}{\sqrt{2}} \left(\begin{smallmatrix}
 i & -1 \\
 1 & -i
\end{smallmatrix}
\right)$ we have
\[
\frac{1}{\sqrt{2}}
\begin{pmatrix}
 i & -1 \\
 1 & -i
\end{pmatrix}
\begin{pmatrix}
ia	&	x	\\
-x	&	ia
\end{pmatrix}
\frac{1}{\sqrt{2}}
\begin{pmatrix}
 -i & 1 \\
 -1 & i
\end{pmatrix}
=
\begin{pmatrix}
 i (a-x) & 0 \\
 0 & i (a+x)
\end{pmatrix}
\]
On the other hand if $a \neq b$ then taking $D = \left(\begin{smallmatrix} -\lambda  & i \mu  \\ -i \mu  & \lambda \end{smallmatrix}\right)$ where $\lambda, \mu \in \R$ such that $\lambda^2 + \mu^2 = 1$ and $\lambda \mu (a-b) = x(\lambda^2 - \mu^2)$, we have
\[
\begin{pmatrix}
 -\lambda  & i \mu  \\
 -i \mu  & \lambda
\end{pmatrix}
\begin{pmatrix}
ia	&	x	\\
-x	&	ib
\end{pmatrix}
\begin{pmatrix}
 -\lambda  & i \mu  \\
 -i \mu  & \lambda 
\end{pmatrix}
=
\begin{pmatrix}
i(a\lambda^2 + b\mu^2 +2x\lambda \mu)	&	0	\\
0	&	i(a\mu^2 + b\lambda^2 - 2x\lambda \mu)
\end{pmatrix}
\]
So we may always assume that, up to conjugation, every $1$-parameter group of $U(2)$ is generated by $\left(\begin{smallmatrix}i\alpha	&	0	\\ 0		&	i\beta \end{smallmatrix}\right)$
with $\alpha, \beta \in \R$. We note that we can interchange $\alpha$ and $\beta$ by conjugation. Via the exponential map this group becomes in 
$
t \mapsto
\left(
\begin{smallmatrix}
e^{it\alpha}	&	0	\\
0				&	e^{it\beta}
\end{smallmatrix}
\right)
$.

Finally if $\alpha = \beta = 0$ then we get the trivial group, if $\beta \neq 0$ we can reparametrize $t \mapsto t/\beta$ obtaining (i) if $\alpha = 0$ and (ii) if $\alpha \neq 0$. Both groups (i) and (ii) are not conjugated because their determinants do not coincide.  
\end{proof}

Among the two types of groups describe in the previous lemma the only $1$-parameter group of isometries of $U(2)$ which fix a curve is $\mathrm{Rot}$. It fixes the set $\ell = \{(z, 0) \in \s^3\}$ which is a great circle that we shall call in the sequel the \emph{axis of rotation}. The other type of group (i) is, for $\alpha = 1$, the traslation along the fiber and, for $\alpha \neq 1$, the composition of a rotation and translation along the fiber.
\medskip

In the Berger sphere $\s^3_b(\kappa, \tau)$, we will denote by $E^1_{(z, w)} = (-\bar{w},\bar{z})$ and $E^2_{(z, w)} = (-i\bar{w}, i \bar{z})$. Then $\{E^1,E^2,V\}$ is an orthogonal basis of $T\s^3_b(\kappa, \tau)$ which satisfies $\lvert E^1\rvert^2 = \lvert E^2\rvert^2 = 4/\kappa$ and $\abs{V}^2 = 16\tau^2/\kappa^2$. The connection $\nabla$ associated to $g$ is given by
\begin{align*}
  \nabla_{E_1} E_1 &= 0, &\nabla_{E_1} E_2 &= -V, &\nabla_{E_1} V &= \frac{4\tau^2}{\kappa}E_2 \\
  \nabla_{E_2} E_1 &= V, &\nabla_{E_2} E_2 &= 0, &\nabla_{E_2} V &= - \frac{4\tau^2}{\kappa} E_1 \\
  \nabla_{V} E_1 &= \left(\frac{4\tau^2}{\kappa} - 2\right)E_2, &\nabla_{V} E_2 &= -\left(\frac{4\tau^2}{\kappa} - 2\right)E_1, &\nabla_{V} V &= 0
\end{align*}

Let $\Phi: \Sigma \rightarrow \s^3_b(\kappa, \tau)$ be an immersion of an oriented constant mean curvature surface $\Sigma$ invariant by $\rot$. Then we can identify $\s^3_b(\kappa, \tau)/\rot$ with $\s^2$ and so $\Sigma$ is $\pi^{-1}(\gamma)$ for some smooth curve $\gamma \subseteq \s^2$. It is sufficient to consider that $\gamma$ is in the upper half sphere and it is parametrized by arc length in $\s^2$, i.e., $\gamma(s) = \bigl(\cos x(s) e^{iy(s)}, \sin x(s) \bigr)$, with $\cos x(s) > 0$ and $x'(s)^2 + y'(s)^2\sin^2 x(s) = 1$ for all $s \in I$. Then we can write down the immersion as $\Phi(s, t) = \bigl(\cos x(s) e^{iy(s)}, \sin x(s) e^{it} \bigr)$. A unit normal vector along $\Phi$ is given by
\[
 N = C \left\{ - \tau \pRe \left[\left(\frac{\tan \alpha}{\cos x} + i \tan x \right) e^{i(t + y)}\right] E^1_\Phi - \tau \pIm \left[\left(\frac{\tan \alpha}{\cos x} + i \tan x \right) e^{i(t + y)}\right] E^2_\Phi + \frac{\kappa}{4\tau}V_\Phi \right\}
\]
where $\alpha$ is an auxiliary function defined by $\cos \alpha(s) = x'(s)$, and 
\[
 C(s) = \frac{\cos x(s) \cos \alpha(s)}{\sqrt{\cos^2 \alpha(s) \bigl[\cos^2 x(s) + \frac{4\tau^2}{\kappa} \sin^2 x(s) \bigr] + \frac{4\tau^2}{\kappa} \sin^2 \alpha(s)}}.
\]

Now by a straigthforward computation we obtain the mean curvature $H$ of $\Sigma$ with respect to the normal $N$ defined above:
\[
 \frac{2\cos^3 \alpha \cos^3 x}{\tau C^3}H = \left(\cos^2 x + \frac{4\tau^2}{\kappa}\sin^2 x\right)\alpha' + \frac{\sin \alpha}{\tan x}\left[ \left(1 - \frac{4\tau^2}{\kappa} \right)\cos^2 x \cos^2 \alpha + \frac{4\tau^2}{\kappa} (1- \tan^2 x) \right] 
\]
Then we get the following result:

\begin{lemma}
 The generating curve $\gamma(s) = \bigl( \cos x(s) e^{iy(s)}, \sin x(s) \bigr)$ of a surface $\Sigma$ of $\s^3_b$ invariant by the group $\rot$ satisfies the following system of ordinary differential equations:
	\begin{equation}\label{eq:sistema}
	 \left\{
		\begin{aligned}
	  x' &= \cos \alpha, \\
		y' &= \dfrac{\sin \alpha}{\cos x}, \\
		\alpha' &= \frac{1}{(\cos^2 x + \frac{4\tau^2}{\kappa} \sin^2 x)}\left\{ \frac{2\cos^3 \alpha \cos^3 x}{\tau C^3}H + \right.\\
		&\left.\qquad -\frac{\sin \alpha}{\tan x}\left[ \left(1 - \frac{4\tau^2}{\kappa} \right)\cos^2 x \cos^2 \alpha + \frac{4\tau^2}{\kappa} (1- \tan^2 x) \right]\right\}
		\end{aligned} 
	 \right.
	\end{equation}
	where $H$ is the mean curvature of $\Sigma$ with respect to the normal defined before. Moreover, if $H$ is constant then the function:
	\begin{equation}\label{eq:integral-primera}
		\tau C \sin x \tan \alpha - H \sin^2 x
	\end{equation}
	is a constant $E$ that we will call the \emph{energy} of the solution.
\end{lemma}

\begin{remark}\label{rmk:properties-system-sphere} From the uniqueness of the solutions of~\eqref{eq:sistema} for a given initial conditions one can show that if $(x, y, \alpha)$ is a solution then:
\begin{enumerate}[(i)]
 \item We can translate the solution by the $y$-axis, i.e., $(x, y + y_0, \alpha)$ is a solution for any $y_0 \in \R$.
 \item Reflection of a solution curve across a line $y = y_0$ is a solution curve with opposite sign of $H$, that is, $(x, 2y_0 - y, -\alpha)$ is a solution for $-H$.
 \item Reversal of parameter for a solution is a solution with opposite sign of $H$, that is, $\bigl(x(2s_0-s), y(2s_0-s), \alpha(2s_0-s) + \pi\bigr)$ is a solution for $-H$. 
 \item If $(x,y, \alpha)$ is defined for $s \in ]s_0 - \varepsilon, s_0 + \varepsilon[$ with $x'(s_0) = 0$ then the solution can be continued by reflection across $y = y(s_0)$.
\end{enumerate}
So thanks to the above properties we can always consider a solution $(x, y, \alpha)$ with positive mean curvature and initial condition $(x_0, 0, \alpha_0)$ at $s = 0$.
\end{remark}

\begin{lemma}\label{lm:restricciones-E-Berger}
Let $\bigl(x(s), y(s), \alpha(s) \bigr)$ be a solution of~\eqref{eq:sistema} with energy $E$. Then the energy $E$ satisfies
\begin{equation}\label{eq:restricciones-energia-berger}
-H - \frac{1}{2}\sqrt{4H^2 + \kappa} \leq 2E \leq -H + \frac{1}{2}\sqrt{4H^2 + \kappa}
\end{equation}
and $x(s) \in [x_1, x_2]$ where $x_j = \arcsin \sqrt{t_j}$, $j = 1, 2$,
\[
t_1 = \frac{\kappa - 8HE - \sqrt{\kappa^2 - 16\kappa E(H+E)}}{2(4H^2 + \kappa)}, \quad t_2 = \frac{\kappa - 8HE + \sqrt{\kappa^2 - 16\kappa E(H+E)}}{2(4H^2 + \kappa)}.
\]
Also $x'(s) = \cos \alpha(s) = 0$ if, and only if, $x(s)$ is exactly $x_1$ or $x_2$.
\end{lemma}

\begin{proof}
First from~\eqref{eq:integral-primera} we obtain
\begin{equation}\label{eq:sin-alpha}
	 \sin \alpha = \frac{1}{\rho} (E + H \sin^2 x) \sqrt{1 + \frac{4\tau^2}{\kappa} \tan^2 x }, \quad
	 \cos \alpha = \frac{1}{\rho}\tau \sin x \sqrt{1 - \frac{4}{\kappa} \frac{(E + H \sin^2 x)^2}{\sin^2 x \cos^2 x}}
\end{equation}
where $\rho = \sqrt{\tau^2 \sin^2 x + \left(1 - \frac{4\tau^2}{\kappa} \right) (E + H \sin^2 x)^2}$. Then $\frac{4}{k}(E + H \sin^2 x)^2 -\sin^2 x \cos^2 x \leq 0$, that is, $p(\sin^2 x) \leq 0$, where $p$ is the polynomial
\[
 p(t) = \left(1 + \frac{4H^2}{\kappa}\right) t^2 - \left(1 - \frac{8H}{\kappa} E \right) t + \frac{4}{\kappa}E^2
\]
As $p(t)$ must be non-positive the vertex of this parabola must be non-positive too, that is, 
\begin{equation}\label{eq:inequality-for-E-H}
 \left(1 - \frac{8}{\kappa}E \right)^2 - \frac{16}{\kappa}E^2 \left(1 + \frac{4H^2}{\kappa} \right) \geq 0
\end{equation}
and $\sin^2 x(s) \in [t_1, t_2]$ where $t_1$ and $t_2$ are the roots of $p$. Finally, as $\cos x(s) \geq 0$ because we choose the curve $\gamma$ on the upper half semisphere, it must be $x(s) \in [0, \pi/2]$ so $x(s) \in [x_1, x_2]$ where $x_j = \arcsin \sqrt{t_j}$, $j = 1, 2$.
\end{proof}

Now we describe the complete solutions of~\eqref{eq:sistema} in terms of $H$ and $E$.

\begin{theorem}\label{tm:clasificacion-berger}
 Let $\Sigma$ be a complete, connected, rotationally invariant surface with constant mean curvature $H$ and energy $E$ in  $\s^3_b(\kappa, \tau)$. Then $\Sigma$ must be of one of the following types:
	\begin{enumerate}[(i)]
	\item If $E = 0$ then $\Sigma$ is a $2$-sphere (possibly immersed, see Corollary~\ref{cor:sphere-CMC-berger}). Moreover, if $H = 0$ too then $\Sigma$ is the great $2$-sphere $\{(z, w) \in \s^3:\, \pIm(z) = 0\}$ which is always embedded.
	 \item If $E = \frac{1}{4}(-2H \pm \sqrt{4H^2 + \kappa})$ then $\Sigma$ is the Clifford torus with radii $r = \sqrt{\frac{1}{2} \pm \frac{H}{\sqrt{4H^2 + \kappa}}}$, that is, $\mathcal{T}_H = \{(z, w) \in \s^3:\, \abs{z} = r^2, \, \abs{w}^2 = 1- r^2\}$.
	 \item If $E > 0$ or $E < -H$ (and different from the case (ii)) then $\Sigma$ is an unduloid-type surface (see figure~\ref{fig:unduloide}).
	 \item If $-H < E < 0$ then $\Sigma$ is a nodoid-type surface (see figure~\ref{fig:nodoide}).
	 \item If $E = -H$ then $\Sigma$ is generated by an union of curves meeting at the north pole (see figure~\ref{fig:circles}).
	\end{enumerate}
	Surfaces of type (iii)--(v) are compact if and only if
	\begin{equation}
	T(H,E) = 2\int_{x_1}^{x_2} \frac{(E+H \sin^2 x) \sqrt{1 + \frac{4\tau^2}{\kappa} \tan^2 x}}{\tau \sqrt{\sin^2 x \cos^2 x - \frac{4}{\kappa} (E+H\sin^2 x)^2}}\,\mathrm{d}x \label{eq:periodo-berger}
	\end{equation}
	is a rational multiple of $\pi$ (see Lemma~\ref{lm:restricciones-E-Berger} for the definition of $x_j$, $j = 1, 2$). Moreover, surfaces of type (iii) are compact and embedded if and only if $T = 2\pi/k$ with $k \in \mathbb{Z}$.
\end{theorem}

\begin{remark}\label{rm:clasificacion-berger}~
\begin{enumerate}
	\item In the round sphere case this study was made by Hsiang~\cite[Theorem 3]{H}. However he did not distinguish, in terms of the energy, between the nodoid and unduloid case. The subriemannian case, which we can think as fixing $\kappa = 4$ and taking $\tau \rightarrow \infty$, was study by Hurtado and Rosales~\cite[Theorem 6.4]{HR} 
	
	\item As $T(H,E)$ is a non-constant continuous function over a non-empty subset of $\R^2$ (see~\eqref{eq:restricciones-energia-berger} for the restrictions of $E$), there exist values of $H$ and $E$ such that $T(H,E)$ is a rational multiple of $\pi$ and so the corresponding surfaces of type (iii)-(v) are compact. 
	
	Among all these compact examples, the minimal ones only appear in (iii) and, from~\eqref{eq:restricciones-energia-berger}, for $0 < E^2 \leq \kappa/16$. For $\kappa = 4$ and $\tau = 0.4$, figure~\ref{fig:periodo-unduloide-minimal} shows that there exists a value of $E$ such that $T(0,E) = 2\pi$, that is, the corresponding surface is embedded and compact so it is an embedded minimal torus which is not a Clifford torus. This surface is a counterexample to the Lawson's conjecture in the Berger sphere $\s^3_b(4,0.4)$. 
	
	The author thinks that there exists a value $\tau_0 \approx 0.57$ such that for $\tau \leq \tau_0$ there are always examples of compact embedded minimal tori (unduloid-type surface) whereas for $\tau > \tau_0$ there are not. These surfaces would be counterexamples to the Lawson's conjecture in the Berger spheres with $\kappa = 4$ and $\tau \leq \tau_0$.
	
	\begin{figure}[htbp]
		\centering
		\includegraphics{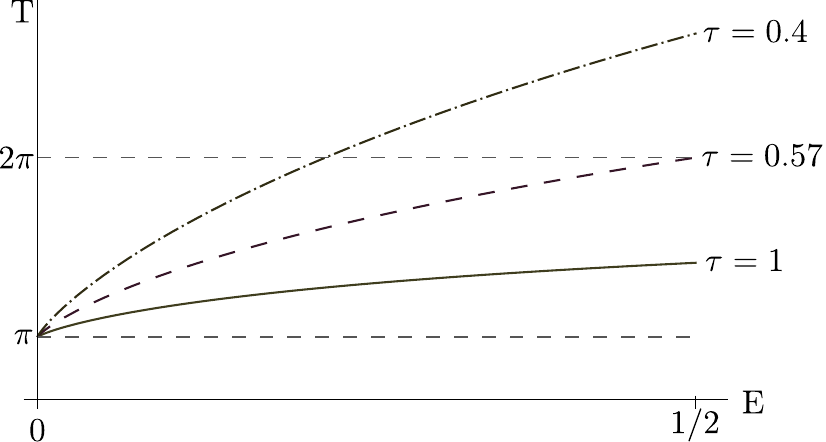}
		\caption{The period $T(0,E)$ (see~\eqref{eq:periodo-berger}) of a minimal unduloid-type surface in terms of the energy $E$ for three differents values of $\tau$ and fixing $\kappa = 4$. We have only depicted the period for $0 < E \leq 1/2$. \label{fig:periodo-unduloide-minimal}}
	\end{figure}
\end{enumerate}
\end{remark}

\begin{proof}
First we obtain several usefull formulae. Sustituting~\eqref{eq:sin-alpha} in the third equation of~\eqref{eq:sistema} we get
\begin{equation}\label{eq:alpha'}
 \alpha'(s) = \frac{\tau^2 \tan x(s) q(\sin^2 x(s))}{\cos x(s) \sqrt{\cos^2 x(s) + \frac{4\tau^2}{\kappa}\sin^2 x(s)}\left[\tau^2 \sin^2 x(s) + \left(1 - \frac{4\tau^2}{\kappa} \right)(E + H \sin^2 x(s))^2\right]^{3/2}}
\end{equation}
where $q(t)$ is the polynomial given by
\begin{equation}\label{eq:alpha'-polinomio}
	\begin{split}
 q(t) &= \frac{H}{\kappa^2}(\kappa - 4\tau^2)(4H^2 + \kappa)t^3 + \frac{1}{\kappa}(\kappa - 4\tau^2) \left( \frac{12EH^2}{\kappa} - (E+2H) \right)t^2 + \\
	&\quad +\left(\frac{12 H E^2\left(\kappa -4 \tau ^2\right)}{\kappa^2}+2 E+H \right) t + \frac{4 E^3 \left(\kappa -4 \tau ^2\right)}{\kappa ^2}-E
	\end{split}
\end{equation} 

\noindent \textbf{(i)} Firstly if $H = 0$ then by~\eqref{eq:sin-alpha} we get that $\sin \alpha = 0$, i.e., $x(s) = s + x_0$ and $y(s) = 0$. Hence the surface $\Sigma$ is the great $2$-sphere $\{(z, w) \in \s^3_b(\kappa, \tau):\, \pIm(z) = 0\}$. Secondly if $H > 0$ then, by Lemma~\ref{lm:restricciones-E-Berger}, $\sin^2 x(s) \in [0, \kappa/(4H^2 + \kappa)]$, i.e., $\tan^2 x(s) \in [0, \kappa/4H^2]$ and we may suppose that $\tan x(s) \in [0, \sqrt{\kappa}/2H]$. By~\eqref{eq:sin-alpha} $\cos \alpha > 0$ in that interval so we can express $y$ as a function of $x$. Taking into account~\eqref{eq:sistema} and~\eqref{eq:sin-alpha} an easy computation shows that
\[
 y'(x) = \frac{H}{\tau} \tan x \frac{\sqrt{1 + \dfrac{4\tau^2}{\kappa} \tan^2 x}}{\sqrt{1 - \dfrac{4H^2}{\kappa}\tan^2x}}, \quad x \in\, \left]0, \arctan \frac{\sqrt{\kappa}}{2H}\right[
\] 
We can integrate the above equation by the change of variable given by 
\[
u = \sqrt{1 - \frac{4H^2}{\kappa}\tan^2 x}\left/\sqrt{1 + \frac{4\tau^2}{\kappa}\tan^2 x}\right..
\]Finally we get
\begin{equation}\label{eq:def-y-sphere}
y(x) = 
\begin{cases}
-\arctan \left(\frac{\tau}{H} \lambda(x)\right) + \frac{H}{\tau} \dfrac{\sqrt{4\tau^2 - \kappa}}{\sqrt{4H^2+\kappa}}\arctan \left(\dfrac{\sqrt{4\tau^2 - \kappa}}{\sqrt{4H^2+\kappa}} \lambda(x)\right) & \text{if } \kappa - 4\tau^2 < 0, \\
\\
-\arctan \left(\frac{\tau}{H} \lambda(x)\right) - \frac{H}{\tau} \dfrac{\sqrt{\kappa - 4\tau^2}}{\sqrt{4H^2+\kappa}}\arctanh \left(\dfrac{\sqrt{\kappa - 4\tau^2}}{\sqrt{4H^2+\kappa}} \lambda(x)\right) & \text{if } \kappa - 4\tau^2 > 0, \\
\end{cases}
\end{equation}
where $\lambda(x) = \left.\sqrt{1 - \frac{4H^2}{\kappa}\tan^2 x}\right/\sqrt{1 + \frac{4\tau^2}{\kappa}\tan^2 x}$. We note that $y\bigl(\arctan(\sqrt{\kappa}/2H)\bigr)=0$ where meets orthogonally the axis $\ell$ and $y$ is a strictly increasing function of $x$, for $x$ in $]0, \arctan(\sqrt{\kappa}/2H)[$. Then $y$ reach its minimun at $x = 0$. The function $y$ only give us half of a sphere, but we can obtain the other half by reflecting the solution along the line $x = 0$. Then its easy to see that the sphere is embedded if, and only if, $y(0) > -\pi$. In other case the sphere is immersed (see figure~\ref{fig:immersed-cmc-sphere}).

\begin{figure}[htbp]
\centering
\includegraphics{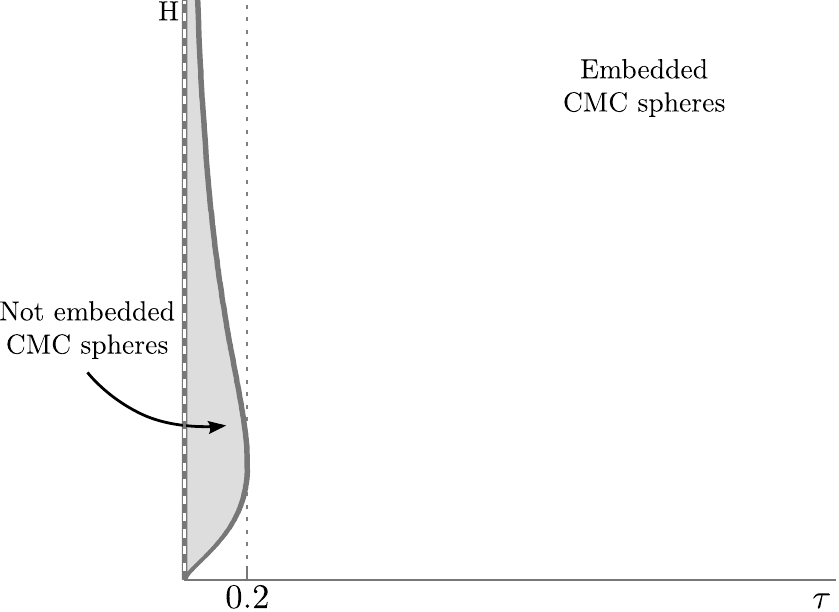}
\caption{Non-embedded region of CMC spheres (we fix $\kappa = 4$)}\label{fig:immersed-cmc-sphere}
\end{figure}

\noindent \textbf{(ii)} If $E = \frac{1}{4}(-2H \pm \sqrt{4H^2 + \kappa})$ the previous lemma says that $t_1 = t_2$ and so $x(s)$ must be the constant $x_1 = x_2 = \arcsin \sqrt{\frac{1}{2}(1 \mp \frac{2H}{\sqrt{4H^2 + \kappa}})}$. We can integrate completely the solution to obtain that $\Phi(s, t) = (r e^{is/r}, \sqrt{1-r^2} e^{it})$ where
\[
 r = \sqrt{\frac{1}{2} \pm \frac{H}{\sqrt{4H^2 + \kappa}}}
\]
i.e., $\Sigma$ is a Clifford torus.
\medskip

\noindent\textbf{[(iii), case $E>0$]}. We suppose now that the equality in~\eqref{eq:inequality-for-E-H} does not hold and that $E > 0$. We consider the maximal solution of~\eqref{eq:sistema} with initial condition $(x_1,0,\pi/2)$ (we will later see that this is not a restriction) and we may suppose, by the maximality condition, that there exist $s_2$ such that $\alpha(s_2) = \pi/2$.
\medskip

We analyze the sign of $\alpha'$ using~\eqref{eq:alpha'}. It is sufficient to study de sign of the polynomial $q$ in~\eqref{eq:alpha'-polinomio} between $t_1$ and $t_2$ (see Lemma~\ref{lm:restricciones-E-Berger}). A straightforward computation shows that $q$ is strictly increasing and that $q(t_1)q(t_2) \leq 0$. Then there exist a unique $s_1$ such that $\alpha'(s_1) = 0$. So $\alpha$ is a strictly increasing function in $]0, s_1[$, strictly decreasing in $]s_1, s_2[$ and $s_1$ is an absolute maximum.  Now, as $\sin \alpha > 0$ we can express $x$ as a function of $y$, then from~\eqref{eq:sistema}
\begin{equation}\label{eq:derivative-x(y)}
 \frac{\mathrm{d}\, x}{\mathrm{d}\, y} = \cos x \cot \alpha > 0
\end{equation}
so $x(y)$ is an strictly increasing function and, because $\cos \alpha(s_2) = 0$, it must be $x(y(s_2)) = x_2$. In particular the solution $x$ takes all values in the interval $[x_1,x_2]$ so, by the unicity of the solution, every maximal solution with initical condicion $(x_0, 0, \alpha_0)$ with $x_0$ necesarilly in $[x_1,x_2]$ must be a reparametrization of this one. Finally, taking into account the above formula, the third equation of~\eqref{eq:sistema} and~\eqref{eq:sin-alpha}, we get
\begin{equation}\label{eq:second-derivative-x(y)}
\begin{split}
\frac{\mathrm{d}^2\, x}{\mathrm{d}\, y^2} = &\frac{-\tau^2 \sin x \cos x}{\left(\cos^2 x + \frac{4\tau^2}{\kappa} \sin^2 x \right)^2 (E + H \sin^2 x)^3}
\left[\cos^2 x\, q(\sin^2 x) +\right. \\
&\quad +\left.\left(\cos^2 x + \frac{4\tau^2}{\kappa}\sin^2 x \right) (E + H \sin^2 x)\left(\cos^2 x \sin^2 x - \frac{4}{\kappa}(E + H \sin^2 x)^2 \right) \right]
\end{split}
\end{equation}
It is straightforwad to check that $\mathrm{d}^2x/\mathrm{d}y^2$ has only one zero at $y_1$ in $]0, y(s_2)[$ and that $x$ is convex in $]0, y_1[$ and concave in $]y_1,y_2[$. By successive reflextions across the vertical lines on which $x(y)$ reaches its critial points, we get the full solution which is similar to an Euclidean unduloid (se figure~\ref{fig:unduloide}).

The period of this unduloid is given by
\begin{equation}\label{eq:periodo}
T = 2y(s_2) = 2\int_{x_1}^{x_2} y'(x) \,\mathrm{d}x 
= 2\int_{x_1}^{x_2} \frac{(E+H \sin^2 x) \sqrt{1 + \frac{4\tau^2}{\kappa} \tan^2 x}}{\tau \sqrt{\sin^2 x \cos^2 x - \frac{4}{\kappa} (E+H\sin^2 x)^2}}
\end{equation}
Hence if~\eqref{eq:periodo} is a rational multiple of $\pi$ then the surface is compact. Moreover, the surface is embedded if and only if $T = 2\pi/k$ for $k \in \N$.
\medskip

\begin{figure}[htbp]
\begin{minipage}[c]{.45\textwidth}
	\includegraphics[width=\textwidth]{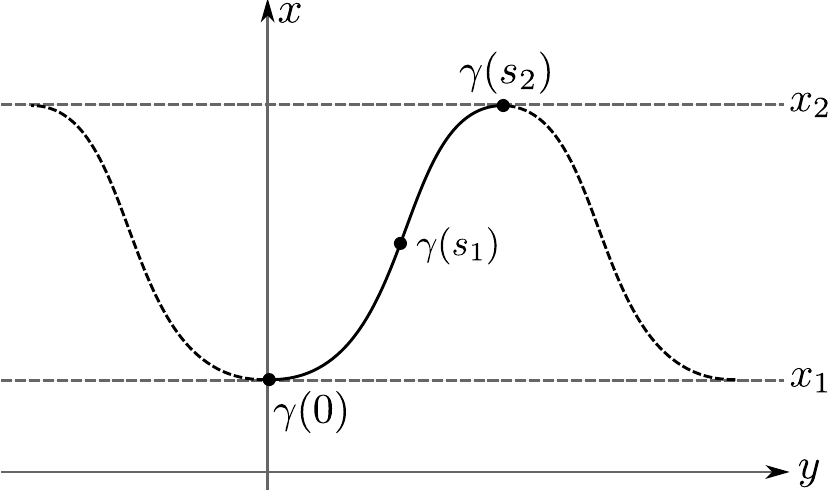}
	\caption{Curve $\gamma(s)$ for $E > 0$\label{fig:unduloide}}
\end{minipage}
\hspace{0.05\textwidth}
\begin{minipage}[c]{0.45\textwidth}
	\includegraphics[width=\textwidth]{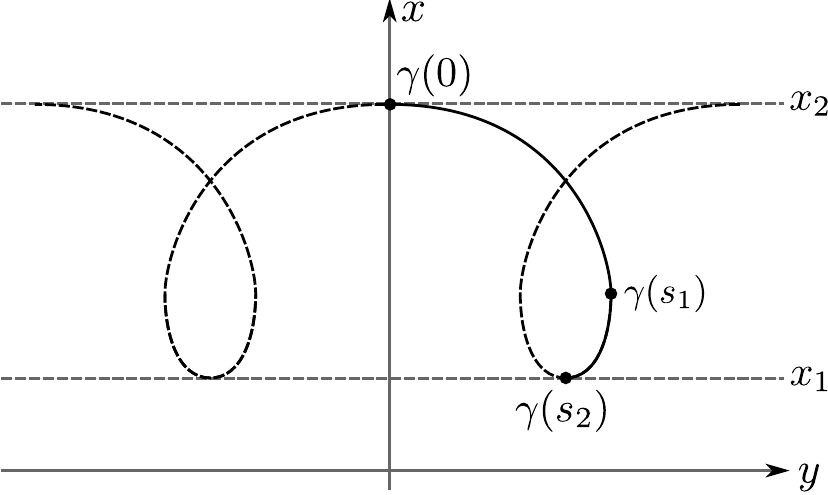}
	\caption{Curve $\gamma(s)$ for $E < 0$ and $E \neq -H$\label{fig:nodoide}}
\end{minipage}
\end{figure}

\noindent\textbf{[(iii), case $E<-H$]}. In this case $\sin \alpha < 0$ so we can express $x$ as a function of $y$ and a similar reasoning as in the previous case is sufficient to check that the surface must be a unduloid (see figure~\ref{fig:unduloide}).
\medskip

\noindent\textbf{(iv)} If $-H < E < 0$ we consider the maximal solution with initial condition $(x_2,0,\pi/2)$. We note that in this case $\sin \alpha$ may change its sign: $\sin \alpha < 0$ if $\sin^2 x \in [t_1, -E/H[$ and $\sin \alpha > 0$ if $\sin^2 x \in ]-E/H, t_2]$. By~\eqref{eq:alpha'} $\alpha' > 0$ so $\alpha$ is strictly increasing. Let $0 < s_1 < s_2$ such that $\alpha(s_1) = \pi$ and $\alpha(s_2) = 3\pi/2$ (and so $x(s_2) = x_1$). Then $\alpha \in ]\pi/2,\pi[$ on $s \in ]0, s_1[$ and $\alpha \in ]\pi, 3\pi/2[$ on $s \in ]s_1,s_2[$. Now we can express the solution $\gamma$ in $]0,s_2[$ as two graphs of the function $x(y)$ meeting at the line $y = y(s_1)$. First using~\eqref{eq:derivative-x(y)} we get that $x(y)$ is strictly decreasing on $]0,y(s_1)[$ and strictly increasing on $]y(s_2), y(s_1)[$. Second taking into account~\eqref{eq:second-derivative-x(y)} $x(y)$ is strictly concave on $]0, y(s_1)[$ and strictly convex on $]y(s_2), y(s_1)[$. As $y = 0$ and $y = y(s_2)$ are lines of symmetry because $x'(0) = x'(s_2) = 0$, we can reflect successively $\gamma$ to obtain the complete solution, which is similar to an Euclidean nodoid (see figure~\ref{fig:nodoide}). Also the solution produce a compact surface if~\eqref{eq:periodo} is a rational multiple of $\pi$ as in the nodoid case. In this case the surface is always immersed.
\medskip

\noindent\textbf{(v)} Finally we study the case $E = -H \neq 0$. Now $\sin x \in [2H/\sqrt{4H^2 + \kappa}, 1]$ so the curve may aproach to the north pole $p_N$ of the $2$-sphere. We consider the maximal solution with initial condition $(\arcsin(2H/\sqrt{4H^2 + \kappa}), 0, 3\pi/2)$ and define $s_1 > 0$ the first number such that $\alpha(s_1) = 2\pi$, that is, the first time the curve $\gamma$ meets the north pole. We can express $x$ as a function of $y$ on every connected component of $\gamma\setminus \{p_N\}$ because $\sin \alpha < 0$ away of $p_N$. Using~\eqref{eq:alpha'} we get that $\alpha' > 0$ and so $\alpha \in [3\pi/2, 2\pi]$. Then, taking into account~\eqref{eq:derivative-x(y)} and~\eqref{eq:second-derivative-x(y)}, we obtain that $x(y)$ is strictly decreasing and convex in $]y(s_1), 0[$.

We continue the generating curve to obtain another branch of the graph of the function $x(y)$ meeting the north pole. We observe now that
\[
\hat{x}(s) = x(2s_1 - s), \quad \hat{y}(s) = 2y(s_1) + \pi - y(2s_1-s), \quad \hat{\alpha}(s) = 3\pi - \alpha(2s_1 - s), \quad s \in [s_1, 2s_1]
\]
is a solution of~\eqref{eq:sistema} with energy $E = -H$. That is, the other branch of the solution is just the reflexion of $x(y)$ with respect the line $y = y(s_1) + \pi/2$. By successive reflexions across the critical points of $x$, we obtain the full solution (see figure~\ref{fig:circles}).

\begin{figure}[htbp]
\centering
\includegraphics{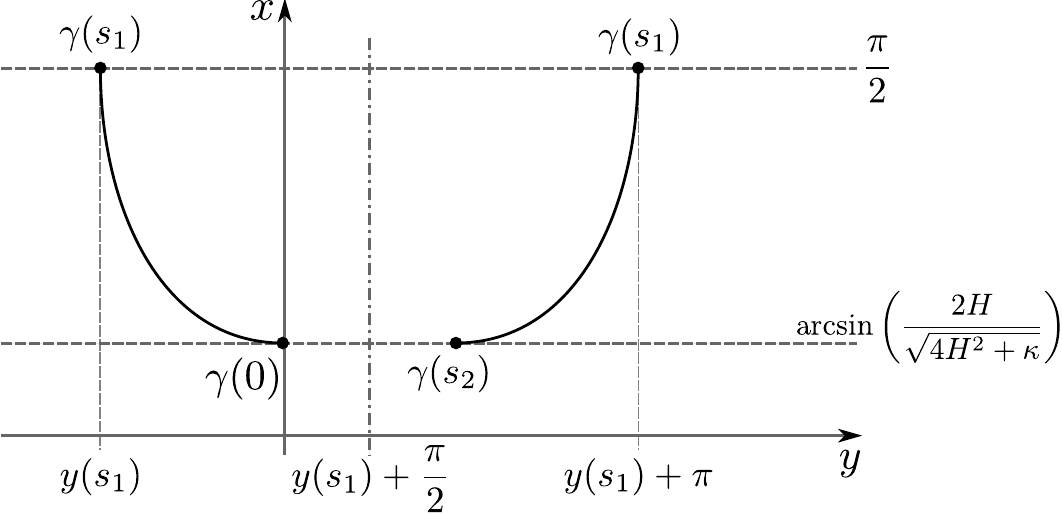}
\caption{Graphic of the curve $\gamma(s) = (x(s), y(s))$ for $E = -H$\label{fig:circles}}
\end{figure}

If $y(s_1) = -\pi/2$ then the reflexion line corresponding to the critical point coincide with the reflexion line of the branch an so the solution will be embedded. Using the next expresion for $y(s_1)$ we have check numerically that for $\tau$ small and for suitable $H$ we get $y(s_1) = -\pi/2$ and so the solution is an embedded torus.
\[
y(s_1) = \int_{\frac{2H}{\sqrt{4H^2 + \kappa}}}^{\pi/2} y'(x) \,\mathrm{d}x 
= \int_{\frac{2H}{\sqrt{4H^2 + \kappa}}}^{\pi/2} \frac{(E+H \sin^2 x) \sqrt{1 + \frac{4\tau^2}{\kappa} \tan^2 x}}{\tau \sqrt{\sin^2 x \cos^2 x - \frac{4}{\kappa} (E+H\sin^2 x)^2}}
\]
Moreover, $\gamma$ is close (and so $\Sigma$ is compact) if, and only if, $y(s_1)$ is a rational multiple of $2\pi$.
\end{proof}

\begin{corollary}\label{cor:sphere-CMC-berger}
Let $\Phi:]-a,a[ \times ]-\pi, \pi[ \rightarrow \s^3_b(\kappa, \tau)$ be the immersion given by:
\[
 \Phi(x, t) = 
	\begin{cases}
	 \left(\cos(x + a) e^{i y(x+a)}, \sin(x + a)e^{i t} \right), & \text{if } x < 0 \\
	 \left(\cos(a - x) e^{-i y(a - x)}, \sin(a - x)e^{i t} \right), & \text{if } x\geq 0 \\
	\end{cases}
\]
where $a = \arctan(\sqrt{\kappa}/2H)$ and $y$ is the function defined in~\eqref{eq:def-y-sphere}. Then $\Phi$ defines an immersion of a sphere with constant mean curvature $H$. Moreover $\Phi$ is an embbedding if and only if $y(0) > -\pi$ (see figure~\ref{fig:immersed-cmc-sphere}).
\end{corollary}

\section{The special linear group $\Sl(2,\R)$}

We are going to study the constant mean curvature surfaces invariante by a $1$-parameter group of isometries in $\Sl(2,\R)$, that is, in the group of real matrix of order $2$ with determinant $1$. It is more convenient to give another description of this group as $\Sl(2,\R) = \{(z, w) \in \C:\, \abs{z}^2 - \abs{w}^2 = 1\}$. It is easy to check that the transformation
\[
\begin{pmatrix}
a	&	b	\\
c	&	d
\end{pmatrix}
\mapsto \frac{1}{2}\Bigl( (a + d) + i (b - c), (b + c) + i(a -d)\Bigr), \quad ad-bc=1
\]
is a diffeomorphism. 
\medskip

We endow $\Sl(2,\R)$ with the metric $g$ given by
\[
g(E^i, E^j) = \delta_{ij}\frac{4}{-\kappa}, \quad g(V,V) = \frac{16\tau^2}{\kappa^2}, \quad g(V,E^j) = 0, \quad i,j = 1, 2.
\]
where $\kappa$ and $\tau$ are real numbers such that $\kappa < 0$ and $\tau \neq 0$ and $\{E^1, E^2, V\}$ is a global reference on $T\Sl(2,\R)$ defined by
\[
E^1_{(z, w)} = (\bar{w}, \bar{z}), \quad E^2_{(z, w)} = (i\bar{w}, i\bar{z}), \quad V_{(z, w)} = (iz, iw)
\]
Then $(\Sl(2,\R), g)$ is a model for an homogeneous space $\E(\kappa, \tau)$ with $\kappa < 0$. $\Sl(2,\R)$ is a fibration over $\mathbb{H}^2(\kappa)$ with fibers generated by the unit killing field $\xi = -\frac{\kappa}{4\tau}V$. We can identify the isometry group of $\Sl(2,\R)$ with $U_1(2)$.
\medskip

The connection associate to $g$ is given by
\begin{align*}
  \nabla_{E_1} E_1 &= 0, &\nabla_{E_1} E_2 &= V, &\nabla_{E_1} V &= \frac{4\tau^2}{\kappa}E_2, \\
  \nabla_{E_2} E_1 &= -V, &\nabla_{E_2} E_2 &= 0, &\nabla_{E_2} V &= - \frac{4\tau^2}{\kappa} E_1, \\
  \nabla_{V} E_1 &= \left(\frac{4\tau^2}{\kappa} - 2\right)E_2, &\nabla_{V} E_2 &= -\left(\frac{4\tau^2}{\kappa} - 2\right)E_1, &\nabla_{V} V &= 0.
\end{align*}

As in the Berger sphere case we concentrate our attention in the $1$-parameter groups of isometries witch fixed a curve, that we call the axis. We define 
\[
\rot = \left\{
		\begin{pmatrix}
			1	&	0	\\
			0	&	e^{it}
		\end{pmatrix}:\, t \in \R
		\right\}
\]
Then $\rot$ fix the curve $\ell = \{(z, 0)\in \Sl(2,\R)\}$ wich is a circle and we can idenfity $\Sl(2,\R)/\rot$ with $O =\{(z, a) \in \C \times \R:\, \abs{z}^2 -a^2 = 1\}$. 
\medskip

Let $\Phi: \Sigma \rightarrow \Sl(2,\R)$ be an immersion of an oriented constant mean curvature surface $\Sigma$ invariant by $\rot$. Then $\Sigma = \pi^{-1}(\gamma)$ for some smooth curve $\gamma \subset O$, where $\pi: \Sl(2,\R) \rightarrow O$ is the projection.

Let $\gamma(s) = (\cosh x(s) e^{iy(s)}, \sinh x(s))$, we may suppose that 
\[
x'(s)^2 + y'(s)^2 \cosh^2 x(s) = 1
\]
and we will call $\alpha$ the function such that $x'(s) = \cos \alpha(s)$. 

Then we can write down the immersion $\Phi(s, t) = (\cosh x(s) e^{iy(s)}, \sinh x(s)e^{it})$. A unit normal vector along $\Phi$ is given by
\[
N = C \left\{ -\tau \pRe\left[\left(\frac{\tan \alpha}{\cosh x} - i\tanh x \right)e^{i(t+y)}\right]E^1_\Phi - \tau \pIm\left[\left(\frac{\tan \alpha}{\cosh x} - i\tanh x \right)e^{i(t+y)}\right] - \frac{\kappa}{4\tau}V_\Phi\right\}
\]
where
\[
C(s) = \frac{\cos \alpha(s) \cosh x(s)}{\sqrt{\cos^2 \alpha(s) [\cosh^2 x(s) - \frac{4\tau^2}{\kappa} \sinh^2 x(s)] - \frac{4\tau^2}{\kappa}\sin^2 \alpha(s)}}
\]

Now by a straightforward computation we get the mean curvature $H$ of $\Sigma$ with respect to the normal defined above:
\[
\begin{split}
\frac{2\cos^3 \alpha \cosh^3 x}{\tau C^3}H &= \left( \cosh^2 x - \frac{4\tau^2}{\kappa}\sinh^2 x\right) \alpha' + \\
&\quad + \frac{\sin \alpha}{\tanh x} \left[ \left(1 - \frac{4\tau^2}{\kappa} \right) \cos^2 \alpha \cosh^2 x + \frac{4\tau^2}{\kappa}(2\cos^2 \alpha - 1)(1 + \tanh^2 x)\right]
\end{split}
\]

Hence we obtain the following result:
\begin{lemma}\label{lm:sistema-Sl(2,R)}
The generating curve $\gamma(s) = (\cosh x(s) e^{iy(s)}, \sinh x(s))$ of a surface $\Sigma$ invariant by the group $\rot$ satisfies the following system of ordinary differential equations:
	\begin{equation}\label{eq:sistema-Sl(2,R)}
	 \left\{
		\begin{aligned}
	  x' &= \cos \alpha, \\
		y' &= \dfrac{\sin \alpha}{\cosh x}, \\
		\alpha' &= \frac{1}{(\cosh^2 x - \frac{4\tau^2}{\kappa} \sinh^2 x)}\left\{ \frac{2\cos^3 \alpha \cosh^3 x}{\tau C^3}H + \right.\\
		&\left.\qquad -\frac{\sin \alpha}{\tanh x}\left[ \left(1 - \frac{4\tau^2}{\kappa} \right) \cos^2 \alpha \cosh^2 x + \frac{4\tau^2}{\kappa}(2\cos^2 \alpha - 1)(1 + \tanh^2 x)\right]\right\}
		\end{aligned} 
	 \right.
	\end{equation}
where $H$ is the mean curvature of $\Sigma$ with respect to the normal defined before. Moreover, if $H$ is constant then the function
\begin{equation}\label{eq:energy-Sl(2,R)}
\tau C \sinh x \tan \alpha - H \sinh^2 x
\end{equation}
is a constant $E$ that we will call the \emph{energy} of the solution.
\end{lemma}

The remark~\ref{rmk:properties-system-sphere} is also true for this system and so we can always consider a solution $(x, y, \alpha)$ with positive mean curvature vector and initial condition $(x_0, 0, \alpha_0)$. 

\begin{lemma}\label{lm:restricciones-E-Sl(2,R)}
Let $(x(s), y(s), \alpha(s))$ be a solution of~\eqref{eq:sistema-Sl(2,R)} with energy $E$. Then:
\begin{enumerate}[(i)]
	\item If $4H^2 + \kappa > 0$ then it must be $4E < 2H - \sqrt{4H^2 + \kappa}$.
	Also $\sinh^2 x(s) \in [t_1, t_2]$ where
	\[
		t_1 = \frac{-8HE - \kappa - \sqrt{16\kappa E(H-E) + \kappa^2}}{2(4H^2 + \kappa)}, \quad 
		t_2 = \frac{-8HE - \kappa + \sqrt{16\kappa E(H-E) + \kappa^2}}{2(4H^2 + \kappa)}
	\]
	Moreover, $x'(s) = \cos \alpha(s) = 0$ if and only if $\sinh^2 x(s)$ is exactly $t_1$ or $t_2$.
	\item If $4H^2 + \kappa < 0$ then $\sinh^2 x(s) \in [t_1, +\infty[$.
	Moreover, $x'(s) = \cos \alpha(s) = 0$ if and only if $\sinh^2 x(s) = t_1$.	
	\item If $4H^2 + \kappa = 0$ then $E < H/2$ and $\sinh^2 x(s) \in [E^2/H(H-2E), +\infty[$. 	Moreover, $x'(s) = \cos \alpha(s) = 0$ if and only if $\sinh^2 x(s) = E^2/H(H-2E)$.
\end{enumerate}
\end{lemma}

\begin{proof}
Using~\eqref{eq:energy-Sl(2,R)} we get that
\begin{equation}\label{eq:sin-alpha-sl}
\sin \alpha = \frac{1}{\mu}(E+H\sinh^2 x) \sqrt{1 - \frac{4\tau^2}{\kappa}\tanh^2 x}, \quad 
\cos \alpha = \frac{1}{\mu}\tau \sinh x \sqrt{1 + \frac{4}{\kappa} \frac{(E+H\sinh^2 x)^2}{\cosh^2 x \sinh^2 x}}
\end{equation}
where $\mu^2 = \tau^2 \sinh^2 x + (E+H\sinh^2 x)^2 \left[1 - \frac{4\tau^2}{\kappa}\left( \tanh^2 x - \frac{1}{\cosh^2 x}\right)\right]$.

From the above formula for $\cos \alpha$ we deduce that $p(\sinh^2 x) \geq 0$, where
\[
p(t) = \left(1 + \frac{4H^2}{\kappa} \right)t^2 + \left(1 + \frac{8HE}{\kappa}\right)t + \frac{4E^2}{\kappa}
\]
The result follows from the study of the sign of this polynomial for $t \geq 0$.
\end{proof}

Now we describe the complete solutions of~\eqref{eq:sistema-Sl(2,R)} in terms of the mean curvature $H$ and the energy $E$.

\begin{theorem}\label{tm:clasificacion-Sl(2,R)}
Let $\Sigma$ be a complete, connected, rotationally invariant surfaces with constant mean curvature $H$ and energy $E$ in $\Sl(2,R)$. Then $\Sigma$ must be one of the following types:
\begin{enumerate}
	\item[1.] If $4H^2 + \kappa > 0$ then
	\begin{enumerate}[(a)]
		\item If $E = 0$ then $\Sigma$ is a $2$-sphere. It is not always embedded (see figure~\ref{fig:sl-embebidas}).
		\item If $E > 0$ then $\Sigma$ is an unduloid-type surface.
		\item If $E < 0$ then $\Sigma$ is a nodoid-type surface which is always immersed.
	\end{enumerate}
	Moreover surfaces of type 1.(b) and 1.(c) are compact if and only if
	\[
		T(H, E) = 2\int_{x_1}^{x_2} \frac{(E+H\sinh^2 x)\sqrt{1-\frac{4\tau^2}{\kappa}\tanh^2 x}}{\tau \sqrt{\sinh^ 2 x \cosh^2 x + \frac{4}{\kappa}(E+H\sinh^ 2 x)^2}}
	\]
	is a rational multiple of $\pi$, where $x_j = \arcsinh \sqrt{t_j}$, $j = 1, 2$ (see Lemma~\ref{lm:restricciones-E-Sl(2,R)}.(i)). Moreover, surfaces of type 1.(b) are compact and embedded if and only if $T = 2\pi/k$ with $k \in \mathbb{Z}$.
	\item[2.] If $4H^2 + \kappa \leq 0$ then $\Sigma$ is immersed and non-compact. Moreover the curve $\gamma$ which generate $\Sigma$ is of the type of figure~\ref{fig:esfera-open} when $E = 0$, figure~\ref{fig:unduloide-open} when $E > 0$ and figure~\ref{fig:nodoide-open} when $E < 0$.
\end{enumerate}
\end{theorem}

\begin{remark}~\label{rm:clasificacion-Sl(2,R)}
\begin{enumerate}
	\item This theorem was first stated by Gorodsky in~\cite{G} for $\kappa = -4$ and $\tau = 1$. However he did not take into account that for $4H^2 +\kappa \leq 0$ there are not constant mean curvature spheres (otherwise, by the Daniel correspondence, see~\cite{D}, we were able to construct constant mean curvature spheres with $4H^2 - 1 \leq 0$ in $\mathbb{H}^2 \times \R$ which is a contradiction by~\cite[Corollary 5.2]{NR}).
	
	\item All the examples described in the above theorem can be lifted to the universal cover. Because the fiber in the universal cover is a line, not a circle, all the constant mean curvature spheres are embedded there. Moreover for $E \geq 0$ the surfaces are embedded too by the same reason. This classification has been obtained, very recently, by Espinoza~\cite{E}.
\end{enumerate}
\end{remark}

\begin{proof}
Firsly we are going to analyze $4H^2 + \kappa > 0$ because it is quite similar to the Berger sphere case. In this case, taking into account the previous lemma and that $H \geq 0$, $x(s)$ moves between two values $x_1 = -\arcsinh \sqrt{t_2}$ and $x_2 = -\arcsinh \sqrt{t_1}$. If $E = 0$ then $x_2 = 0$ and so the curve $\gamma$ may intersect the axis of rotation $\ell$. As $\cos \alpha > 0$ for $x(s) \in ]x_1 = -\arctanh(\sqrt{-\kappa}/2H), 0[$ we can express $y$ as a function of $x$. Now using~\eqref{eq:sin-alpha-sl} we get that
\[
y'(x) = \frac{H \tanh x \sqrt{1 - \frac{4\tau^2}{\kappa} \tanh^2 x}}{\tau \sqrt{1 + \frac{4H^2}{\kappa}\tanh^2 x}}
\]
And we can integrate explicitly this equation to obtain that
\begin{equation}\label{eq:def-sphere-sl}
y(x) = \arctan \left(\frac{\tau}{H} \rho(x)\right) - \frac{H}{\tau} \dfrac{\sqrt{4\tau^2 - \kappa}}{\sqrt{4H^2+\kappa}}\arctan \left(\dfrac{\sqrt{4\tau^2 - \kappa}}{\sqrt{4H^2+\kappa}} \rho(x)\right)
\end{equation}
where $\rho(x) = \left.\sqrt{1 + \frac{4H^2}{\kappa}\tanh^2 x}\right/\sqrt{1 - \frac{4\tau^2}{\kappa}\tanh^2 x}$. We note that $y(x_1) = 0$ where meets orthogonally the axis $\ell$ and $y(x)$ is a strictly increasing and strictly convex. The function $y(x)$ only describe half of the sphere, but we can obtain the whole sphere by reflection the solution along the line $x = 0$. It is easy to see that the sphere is embedded if, and only if, $y(0) > -\pi$ (see the figure~\ref{fig:sl-embebidas}).

\begin{figure}[htbp]
\centering
\includegraphics[width=0.5\textwidth]{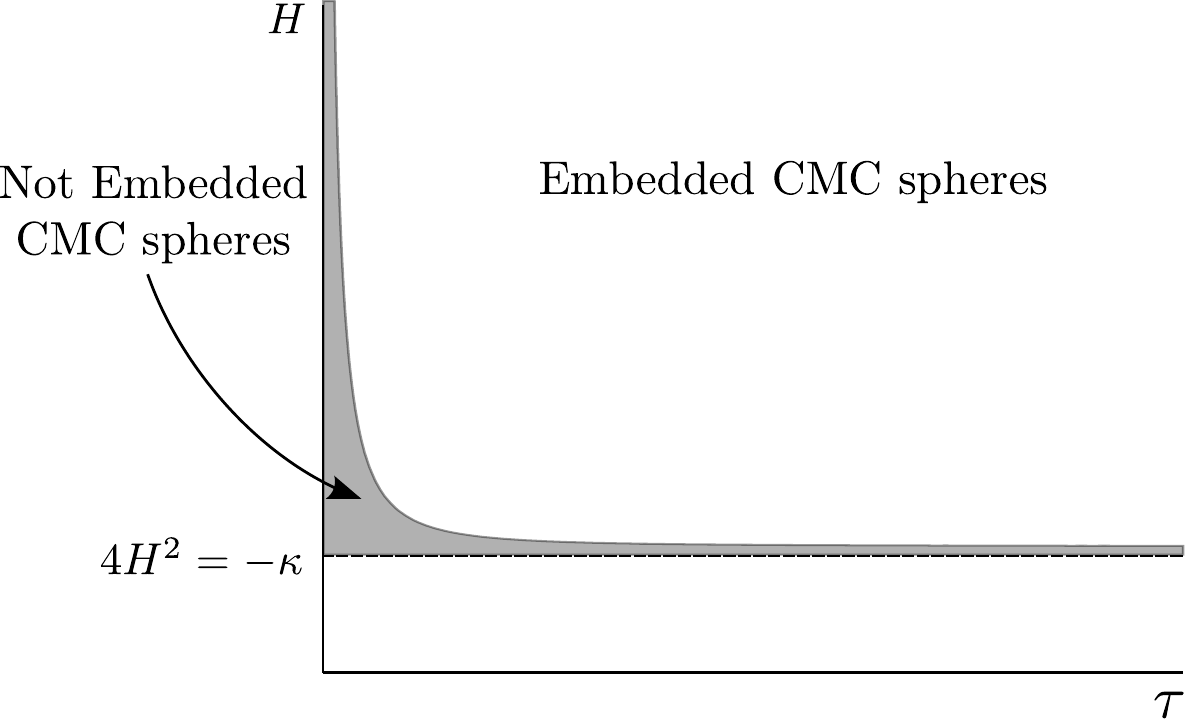}
\caption{Non-embedded region of CMC spheres in $\Sl(2,\R)$.\label{fig:sl-embebidas}}
\end{figure}

Now if $E > 0$ then $\sin \alpha > 0$ by~\eqref{eq:sin-alpha-sl} and so we can express $x$ as a function of $y$. A similar reasoning as in the Berger sphere case for $E > 0$ is sufficient to check that the surface must be a unduloid (see figure~\ref{fig:unduloide}). Finally if $E < 0$ then $\sin \alpha$ may change its sign. As in the Berger sphere case for $-H < E < 0$ we can express the curve $(x(s), y(s))$ as two graphs of the function $x(y)$. Hence it is straightforward to check that the situation is the same as in figure~\ref{fig:nodoide} and the surface must be a nodoid-type one. In both cases the surface is compact if and only if
\begin{equation}\label{eq:periodo-Sl(2,R)}
T(H,E) = 2\int_{x_1}^{x_2} y'(x) \df x = 2 \int_{x_1}^{x_2} \frac{(E+H\sinh^2 x)\sqrt{1-\frac{4\tau^2}{\kappa}\tanh^2 x}}{\tau \sqrt{\sinh^ 2 x \cosh^2 x + \frac{4}{\kappa}(E+H\sinh^ 2 x)^2}}
\end{equation}
is a rational multiple of $\pi$, where $x_j = \arcsinh \sqrt{t_j}$, $j = 1, 2$ (see Lemma~\ref{lm:restricciones-E-Sl(2,R)}).
\medskip

On the other hand the situation for $4H^2 + \kappa \leq 0$ is different from the above and it does not have a counterpart in the Berger sphere case. We firstly observe, by the previous lemma, that in this case $x(s)$ does not have to move between two real values. It is only bounded above by a constant that depend on $H$ and only vanish when $E = 0$ so the solution intersect the axis of rotation only in this case. Moreover, as $x'(s) = \cos \alpha(s)$ can only vanish once the solution cannot be periodic. We are going to distinguish between $E = 0$, $E > 0$ and $E < 0$ and we define for all the cases $x_1 = -\arcsinh\sqrt{t_1}$ when $4H^2 + \kappa < 0$ and $x_1 = -\arcsinh (\abs{E}/\sqrt{H(H-2E)})$ when $4H^2 + \kappa = 0$. Because we choose $H \geq 0$ it must be $x(s) \in ]-\infty, x_1]$. 

If $E = 0$ then we consider the maximal solution with inicial condition $(0, 0, \pi)$. In this case $\cos \alpha(s) < 0$ for any $s$ so we can express $y$ as a function of $x$. Then
\[
\frac{\mathrm{d}\,y}{\mathrm{d}\, x} = \frac{H \sinh x \sqrt{1 - \frac{4\tau^2}{\kappa}\tanh^2 x}}{\sqrt{1 + \frac{4H^2}{\kappa} \tanh^2 x}} < 0
\]
so the function $y$ is strictly decreasing. Moreover $\mathrm{d}^2\, y/\mathrm{d}\, x^2 > 0$ so the function $y$ is strictly convex. In figure~\ref{fig:esfera-open} we can see the two situations for $E = 0$.
\medskip

\begin{figure}[htbp]
\centering
\includegraphics{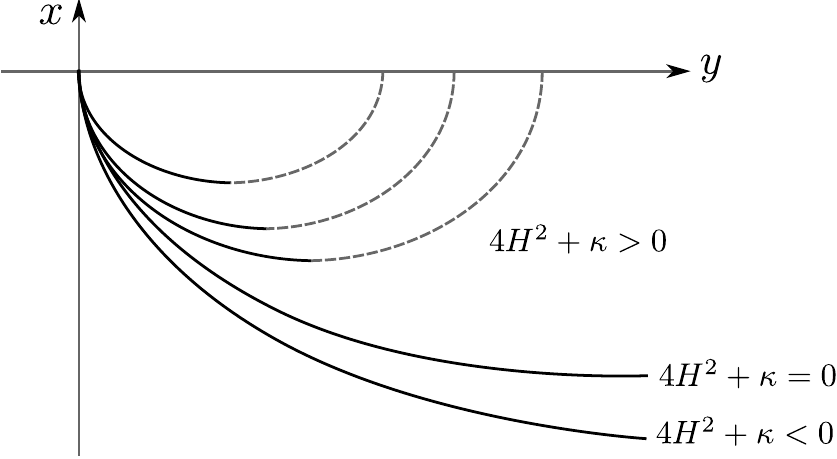}
\caption{Different solutions for $E = 0$ depending on the sign of $4H^2 + \kappa$\label{fig:esfera-open}}
\end{figure}

In the second case, that is for $E > 0$, we consider the maximal solution with initial conditions $(x_1, 0, \pi/2)$. Then there exists $s_1 > 0$ such that $\alpha'(s_1) = 0$, $\alpha'$ is positive for $s < s_1$ and negative for $s>s_1$. Hence, using that $\cos \alpha < 0$ and $\sin \alpha > 0$ by~\eqref{eq:sin-alpha-sl}, we get that $\alpha(s) \in ]\pi/2, \pi[$. We can express $x$ in terms of $y$ because $\sin \alpha > 0$ by~\eqref{eq:sin-alpha-sl} and
\begin{equation}\label{eq:dx/dy-Sl(2,R)}
\frac{\df x}{\df y} = \cot \alpha \cosh x < 0
\end{equation}
so $x$ is a strictly decreasing function of $y$ (see figure~\ref{fig:unduloide-open}).
\medskip

Finally when $E < 0$ we consider the maximal solution with initial condition $(x_1, 0, 3\pi/2)$. In this case $\sin \alpha$ could vanish so we can not express $x$ as a function of $y$. As $\alpha'$ is always negative let $s_1 > 0$ such that $\alpha(s_1) = \pi$. Then $\alpha \in ]\pi, 3\pi/2[$ on $s \in ]0, s_1[$ and $\alpha \in ]\pi/2, \pi[$ on $s > s_1$ because $\cos \alpha (s)$ does not vanish anymore. Now we can express the solution $\gamma$ as two graphs of the function $x(y)$ meeting at the line $y = y(s_1)$. First using~\eqref{eq:dx/dy-Sl(2,R)} $x(y)$ is strictly increasing on $]y(s_1), 0[$ and strictly decreasing on $]y(s_1), +\infty[$. Therefore the solution must be similar to the figure~\ref{fig:nodoide-open}. 
\begin{figure}[htbp]
\begin{minipage}[c]{.45\textwidth}
	\includegraphics[width=\textwidth]{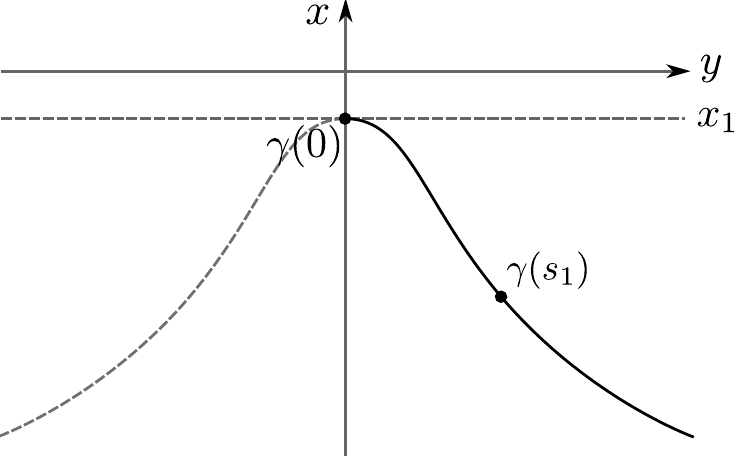}
	\caption{Curve $\gamma(s)$ for $4H^2 + \kappa \leq 0$ and $E > 0$\label{fig:unduloide-open}}
\end{minipage}
\hspace{0.05\textwidth}
\begin{minipage}[c]{0.45\textwidth}
	\includegraphics[width=\textwidth]{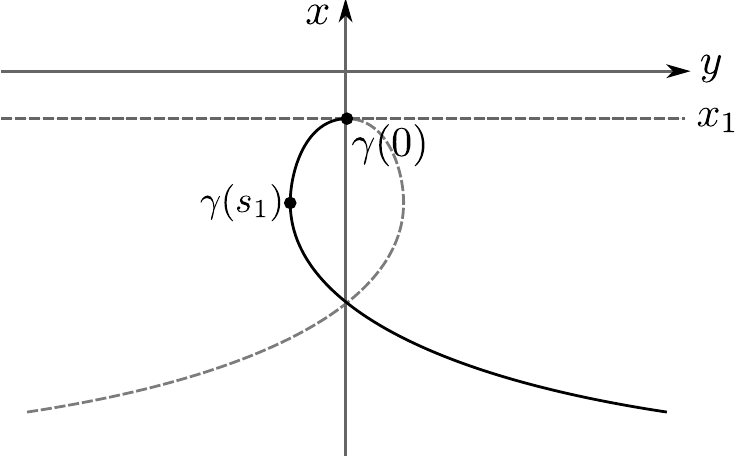}
	\caption{Curve $\gamma(s)$ for $4H^2 + \kappa \leq 0$ and $E < 0$\label{fig:nodoide-open}}
\end{minipage}
\end{figure}
\end{proof}

\begin{corollary}
Let $4H^2 + \kappa > 0$ and $\Phi:]-a,a[ \times ]-\pi, \pi[ \rightarrow \Sl(2,\R)$ be the immersion given by:
\[
 \Phi(x, t) = 
	\begin{cases}
	 \left(\cosh(x + a) e^{i y(x+a)}, \sinh(x + a)e^{i t} \right), & \text{if } x < 0 \\
	 \left(\cosh(a - x) e^{-i y(a - x)}, \sinh(a - x)e^{i t} \right), & \text{if } x\geq 0 \\
	\end{cases}
\]
where $a = \arctanh(\sqrt{-\kappa}/2H)$ and $y$ is the function defined in~\eqref{eq:def-sphere-sl}. Then $\Phi$ defines an immersion of a sphere with constant mean curvature $H$. Moreover $\Phi$ is an embbedding if and only if $y(0) > -\pi$ (see figure~\ref{fig:sl-embebidas}).
\end{corollary}

\section{The isoperimetric problem in the Berger spheres}

In~\cite{TU09-2} the authors studied the stability of constant mean curvature surfaces in the Berger spheres. They proved that for $1/3 \leq 4\tau^2/\kappa < 1$ the solution to the isoperimetric problem are the rotationally constant mean curvature $H$ spheres, $\mathcal{S}_H$. Besides they showed that there exist unstable constant mean curvature spheres for $\tau$ close to zero. Moreover, for $4\tau^2/\kappa < 1/3$ there exist stable constant mean curvature spheres and tori. The aim of this section is to study the relation between the area and the volumen of the rotatilonally constant mean curvature spheres in order to understand the isoperimetric problem for $4\tau^2/\kappa < 1/3$.
\medskip

We have given in Corollary~\ref{cor:sphere-CMC-berger} a parametrization of the constant mean curvature sphere $\mathcal{S}_H$. Then, using that parametrization, we define the interior domain of $\mathcal{S}_H$ as 
\[
\Omega_H = \{(z, w) \in \s^3:\, -y(\arccos\abs{z}) < \arg(z) < y(\arccos\abs{z})\}
\]
Hence one of the volumes determined by $\mathcal{S}_H$ is $\mathrm{vol}(\Omega_H)$ (note that this does not have to be the smaller one).

\begin{lemma}
The volume of $\Omega_H$ is given by:
\[
\mathrm{vol}(\Omega_H) = 
\begin{cases}
	\dfrac{16\pi \tau}{\kappa^2} \left(2\arctan\left( \frac{\tau}{H} \right) - \dfrac{\kappa H}{\tau(4H^2 + \kappa)} + \mu \arctan\left(\frac{\sqrt{4\tau^2 -\kappa}}{\sqrt{4H^2 + \kappa}} \right)\right), & \text{if } \kappa - 4\tau^2 < 0 \\
\\ 
	\dfrac{16\pi \tau}{\kappa^2} \left(2\arctan\left( \frac{\tau}{H} \right) - \dfrac{\kappa H}{\tau(4H^2 + \kappa)} + \mu\arctanh\left(\frac{\sqrt{\kappa - 4\tau^2}}{\sqrt{4H^2 + \kappa}} \right)\right), & \text{if } \kappa - 4\tau^2 > 0 \\	
\end{cases}
\]
where
\[
\mu = \frac{2H}{\tau}\frac{(\kappa - 4\tau^2)(2H^2 + \kappa) - 2\tau^2(4H^2 + \kappa)}{\sqrt{\abs{4\tau^2 - \kappa}}(4H^2 + \kappa)^{3/2}}.
\]
\end{lemma}

\begin{proof}
Firstly it is easy to see that, because the symmetry of the sphere, we can restrict ourselves to the domain $\Omega_H^+ = \{(z, w) \in \s^3:\, \arg(z) < y(\arccos \abs{z})\}$ so $\mathrm{vol}(\Omega_H) = 2\mathrm{vol}(\Omega_H^+)$. Secondly the volume form $\omega_b$ of $\s^3(\kappa, \tau)$ and $\omega$ of $\s^3$ are related by $\omega_b = \frac{16\tau}{\kappa^2}\omega$. Hence it is sufficient to calculate the volume of $\Omega_H^+$ with respecto to the standar metric on the sphere. 

We are going to apply the co-area formula using the function $f(z, w) = \arccos\abs{z}$ which asign to the point $(z, w)$ the distant to the curve $\ell = \{(z, 0) \in \s^3\}$, which is the axis of rotation of the group $\mathrm{Rot}$. Then
\[
\mathrm{vol}(\Omega_H) = 2\int_{\Omega_H^+} \omega_b = \frac{32\tau}{\kappa^2}\int_{\Omega_H^+} \omega = \frac{32\tau}{\kappa^2}\int_\R \left( \int_{\Gamma_t \cap \Omega_H^+} \omega_t\right) \df t
\]
where $\omega_t$ is the restriction of the form $\omega$ to $\Gamma_t = f^{-1}(\{t\})$ and we have taken into account that $\abs{\nabla f} = 1$. Now we can parametrize $\Gamma_t \cap \Omega_H^+$ as $\varphi:[0, y(t)] \times [0, 2\pi] \rightarrow \Gamma_t \cap \Omega_H^+$, $\varphi(u, \theta) = \bigl(\cos t e^{iu}, \sin t e^{i\theta} \bigr)$. We note that $\Gamma_t \cap \Omega_H^+ = \emptyset$ for $t > \arctan(\sqrt{\kappa}/2H)$ or $t < 0$. Hence the above integral can be rewritten as
\[
\begin{split}
\mathrm{vol}(\Omega_H) &= \frac{32\tau}{\kappa^2}\int_{0}^{\arctan(\sqrt{\kappa}/2H)} \left(\int_{[0, y(t)] \times [0, 2\pi]} \varphi^*(\omega_t) \right) \df t = \\
&= \frac{16\pi\tau}{\kappa^2} 4 \int_{0}^{\arctan(\sqrt{\kappa}/2H)} \sin t\, \cos t\, y(t) \df t
\end{split}
\]
Finally a long but straightforward computation yields the above integral and the result.
\end{proof}

Now we are able to draw the area of $\mathcal{S}_H$ in terms of its volume. We are going to compare it with the tori $\mathcal{T}_H$ (see figure~\ref{fig:perfil-esfera-toro}) because for $\tau$ close to zero there are stable constant mean curvature spheres and tori (see~\cite{TU09-2}) so both surfaces are candidates so solve the isoperimetric problem. The area and the smallest volume enclose by $\mathcal{T}_H$  are given by:
\[
\mathrm{Area}(\mathcal{T}_H) = \frac{4\tau}{\kappa}\frac{4\pi^2}{\sqrt{4H^2 + \kappa}},\qquad \text{Volume enclosed by } \mathcal{T}_H = \frac{16\pi^2\tau}{\kappa^2}\left(1 - \frac{2H}{\sqrt{4H^2 + \kappa}} \right)
\]

 \begin{figure}[htbp]
        \begin{minipage}[c]{0.4\textwidth}
         \centering
                \includegraphics[width=\textwidth]{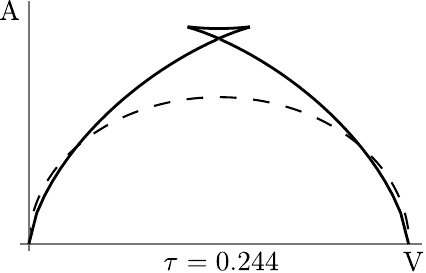}
        \end{minipage}
        \hspace{0.1\textwidth}
        \begin{minipage}[c]{0.4\textwidth}
         \centering
        \includegraphics[width=\textwidth]{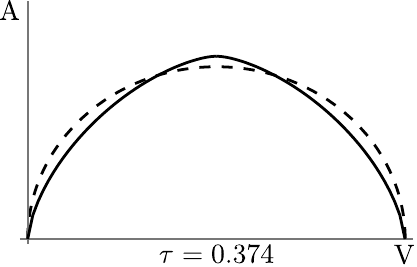}
        \end{minipage}
        \\[0.5cm]
        \begin{minipage}[c]{0.4\textwidth}
         \centering
                \includegraphics[width=\textwidth]{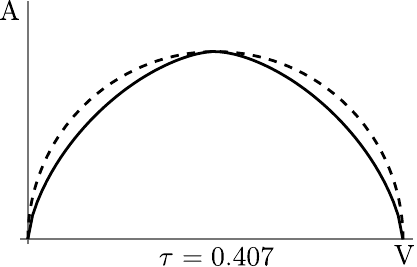}
        \end{minipage}
        \hspace{0.1\textwidth}
        \begin{minipage}[c]{0.4\textwidth}
         \centering
                \includegraphics[width=\textwidth]{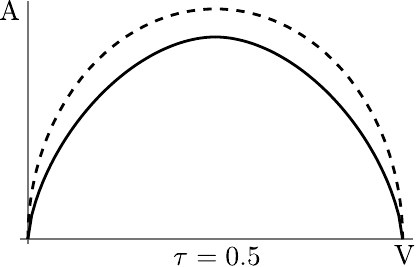}
        \end{minipage}
        \caption{Graphics of the area of the CMC spheres (solid line) and CMC
tori (dashed line) in terms of the volume for diferent Berger spheres ($\kappa = 4$).\label{fig:perfil-esfera-toro}}
 \end{figure}

We can fix, without lost of generality, $\kappa = 4$. Figure~\ref{fig:perfil-esfera-toro} shows the four different situation that appears in the Berger spheres: for $\tau = 0.5$ the spheres are the best candidates to solve the isoperimetric problem, for $\tau = 0.407$ the minimal Clifford torus has the same area and volumen that the minimal sphere so both are candidates to solve the isoperimetric problem. For $\tau = 0.374$ and $\tau = 0.244$ (in the last case there are unsatable spheres and non-congruent spheres enclosing the same volume) it appears an open interval centered at $\pi^2 16 \tau/\kappa^2$ such that the tori $\mathcal{T}_H$ are the candidates to solve the isoperimetric problem.


\end{document}